\documentclass[11pt,a4paper]{article}
\usepackage[utf8]{inputenc}
\usepackage{amsmath}
\usepackage{amsfonts}
\usepackage{amssymb}
\usepackage{graphicx}
\usepackage{amsthm}
\usepackage{hyperref}
\usepackage[left=2.5cm,right=2.5cm,top=2cm,bottom=2cm]{geometry}
\usepackage{authblk}
\usepackage{framed}
\usepackage{listings}
\usepackage{color}
\usepackage{float}

\usepackage{chngcntr}
\counterwithout{figure}{section}

\usepackage{setspace}

\usepackage{color}
\usepackage{amsfonts, amsmath, amsfonts, amssymb}
\usepackage{graphicx}
\usepackage{esint}

\usepackage{caption}

\theoremstyle{note}
\newtheorem{thm}{Theorem}
\newtheorem{lem}[thm]{Lemma}

\newtheorem{defn}{Definition}

\theoremstyle{remark}
\newtheorem*{rem}{Remark}

\begin{document}

\title{Can one see the shape of a network?}
\author{Melanie Weber$^{1,2}$, Emil Saucan$^{1,3}$and J{\"u}rgen Jost $^{1,4}$\\
\footnotesize $^{1}$ Max Planck Institute for Mathematics in the Sciences; Inselstrasse 22, 04103 Leipzig, Germany \\
\footnotesize $^{2}$ University of Leipzig; Augustusplatz 10, 04109 Leipzig, Germany \\
\footnotesize $^{3}$ Technion -- Israel Institute of Technology; Haifa 32000, Israel\\
\footnotesize $^{4}$ Santa Fe Institute; 1399 Hyde Park Road Santa Fe, New Mexico 87501 USA\\
{\footnotesize (Contact: mw25@math.princeton.edu, semil@ee.technion.ac.il, jost@mis.mpg.de)}
}

\maketitle

\begin{abstract}
\noindent Traditionally, network analysis is based on local properties of vertices, like their degree or clustering, and their statistical behavior across the network in question. This paper develops an approach which is different in two respects. We investigate edge-based properties, and we define global characteristics of networks directly. The latter will provide our affirmative answer to the question raised in the title. 
More concretely, we start with Forman's notion of the Ricci curvature of a graph, or more generally, a polyhedral complex. This will allow us to pass from a graph as representing a network to a polyhedral complex for instance by filling in triangles into connected triples of edges and to investigate the resulting effect on the curvature. This is insightful for two reasons: First, we can define a curvature flow in order to asymptotically simplify a network and reduce it to its essentials. Second, using a construction of Bloch, which yields a discrete Gauss-Bonnet  theorem, we have the Euler characteristic of a network as a global characteristic. These two aspects beautifully merge in the sense that the asymptotic properties of the curvature flow are indicated by that Euler characteristic.
\end{abstract}

\section{Introduction}
\noindent The field of Network Science studies a wide range of systems and structures represented as graphs. Numerous methods have been introduced to study their local structure: From clustering coefficients and community detection methods to assortativity and mixing patterns resulting in a variety of network-analytic tools to analyze the local structure of distinguished regions \cite{ab_review,newman2}. However, what can we say about the global structure? If we could take a bird’s eye view we might ask - in paraphrase of \cite{Ka}: \textit{Can we see the shape of a network?}

In \cite{SMJSS} part of the authors introduced Forman-Ricci curvature,  as an edge-based characteristic for networks as classic graphs in one dimension. In \cite{WJS2} the authors extended this study, and investigated its associated geometric flows. The present article extends the formalism therein to higher-order structures by building on previous work of R. Forman \cite{Fo} and E. Bloch \cite{Bloch} on $n$-dimensional cell complexes. In \cite{Fo}, Forman introduced a discretization of Ricci-curvature and its associated flows for CW complexes. The polyhedral complexes considered here, are special cases to which his more general theory applies. While giving a mathematically rigorous formulation that allows for efficient computation, Forman's work unfortunately fails to map essential results from differential geometry, most importantly the \textit{Gau{\ss}-Bonnet Theorem}, to the discrete case. However, in a recent article E. Bloch \cite{Bloch} develops a discrete Gau{\ss}-Bonnet style theorem that can be applied to networks and their extensions to higher dimensional complexes. 

In the present work, we extend the classic graphs by including higher degree faces and extrapolating network graphs built from empirical data to higher dimensional polyhedral complexes. For the special case of unweighted networks, we present a simple formula for the influence of higher degree faces on an edge’s Ricci curvature.
The higher dimensional substructures or \textit{faces} considered here are of great importance for the analysis of real-world networks. They represent strongly associated sets of nodes that are either pairwise connected (\textit{2-faces}) or form a densely interconnected local cluster (\textit{n-faces}). In terms of applications to complex systems, they are especially important in quantitative biology, where correlation networks are widely used -- for instance, co-expression networks in Genomics or brain networks in Neuroscience. In these fields, the investigation of such local substructures is a relevant part of network analysis.

The second part of the article suggests theoretical tools to classify the shape of a given network. We introduce a network-theoretic formulation of the Gau{\ss}-Bonnet theorem for 2-dimensional complexes that allows for computation of the Euler characteristic based on Ricci curvature. This enables us to quantify the shape of a network by computing its distribution of Forman-Ricci curvature. Furthermore, in analogy to the model 
geometries arising in the classical (surface) Ricci flow, we attempt to define {\it prototype} networks, introducing both a classification scheme for network shapes and a tool to study qualitatively the long term behavior and possible limit cases of dynamically evolving networks. With this, we introduce a theoretical foundation for the prediction of future network states and eventually the long term study of dynamic effects in complex systems.


\section{Higher dimensional faces in networks}

\noindent In the first section we extend the widely used concept of the network graph, commonly perceived as a {\em regular, 1-dimensional cellular complex}, by including higher dimensional faces. The resulting 2-dimensional complexes are more complex objects than their classic 1-dimensional counterpart but can be described with similar formalisms. We consider them here as \textit{polyhedral complexes} instead of the more general cellular complexes. This choice imposes some restrictions on the connectivity and degenerated substructures (e.g. loops, multiple edges, isolated edges). Here, we neglect these degenerated cases and assume a connected graph. In practical examples and computational investigations, we consider the largest connected components, if facing the later issue.

Recall the definition of a polyhedral complex \cite{gruenbaum}:
\begin{defn}[polyhedral complex, 2D]
A 2-dimensional polyhedral complex X is a triplet (V, E, F) with $V \neq \varnothing$ a set of nodes (or vertices), $E$ a set of edges  and $F$ a set of faces, such that
\begin{enumerate}
\item each $e \in E$ is incident to two nodes $v_i , v_j \in V$,
\item each $f \in F$ is a polygon with nodes $v_i \in V$ and edges $e_j \in E$ 
\end{enumerate}
\end{defn}
\noindent When confronted with a network, as represented by a graph, we may pass to a higher-dimensional polyhedral complex by inserting polyhedra into certain graph motives, in order to express suitable relations between the vertices involved in a geometrical manner. For instance, whenever we find a cycle of some short length $\ell \le L$, we may insert a polygon into those $\ell$ edges. This would express the fact that such short cycles correspond to groups of vertices that are in some sense closely related to each other. Likewise, whenever we find a complete subgraph $K_p$ for some small $p\le P$, we can insert a $(p-1)$-dimensional simplex. As we shall see, such insertions may have the effect of decreasing the curvature. In some sense, this represents a more refined geometric representation of the close relationships between the vertices involved than simply contracting the subgraph in question to a single vertex. 
In practice this amounts to adding only those simplices that model correlations of higher order (edges encoding correlations between couples of nodes), thus triangles correspond to correlations between 3 nodes (e.g. genes in a genetic regulatory network), 3-dimensional simplices to correlations between 4 nodes, etc. (See the detailed discussion and Tables 1 and 2 in Section 2.) This allows us not only to obtained a more refined analysis of the geometry of underlying network via that of the associated complex, but also -- and perhaps more importantly -- it allows us to, by incorporating in a meaningful manner the higher order correlations, to better model real life networks.


We refer to the resulting $p$-dimensional cells as $p$-faces. Moreover, the resulting $p$-faces are called \textit{parents} of the $(p-1)$-faces (notation: $f^p > f^{p-1}$) and, in turn, the $(p-1)$-faces are the \textit{children} of the $p$-faces (notation: $f^{p-1} < f^{p}$). We introduce a notation for the set of $p$-faces of a graph G as
\begin{eqnarray}
F_p (G) = \lbrace	f_i^p	\rbrace_{i} \; .
\end{eqnarray}
With this, recall the definition of \textit{parallel} faces:
\begin{defn}[parallel faces]
Let $f^p, \hat{f}^p \in F_p(G)$. Then $f^p$ and $\hat{f}^p$ are {\rm parallel} (notation: $f^p \parallel \hat{f}^p$) iff
\begin{enumerate}
\item $\exists f^{p+1}: \; f^p, \hat{f}^p < f^{p+1}$ or
\item $\exists f^{p-1}: \; f^p, \hat{f}^p > f^{p+1}$
\end{enumerate}
but not both.
\end{defn}
\noindent This is illustrated in Fig. \ref{fig:master} where two parallel edges are marked in red. They share a common parent (the gray quadrangle), but distinct children (their nodes).

\subsection{2-faces}

In the specific case $p=2$, we construct \textit{2-faces} $f_d^2$ from simple closed paths between non-interconnected $n$-tuple of nodes. We refer to the index $d$ as the \textit{degree} (or \textit{order}) \textit{of the face}. In the classic $1$-dimensional complex this refers to triangles between triples of edges, quadrangles between quadruples of nodes etc. To preserve the structural information of the 1-dimensional complex, we introduce face weights according to an analogy from classic geometry. In \cite{WJS}, we consider ``default" edge weights from node weights in analogy to the length of a curve from the positions of their end points. For the 2-dimensional case, we construct \textit{face weights} from edge weights by the analogy of area computation, as follows:

\begin{enumerate}

\item \textbf{Triangles} (d=3) \\
We use Heron's formula for the area of a triangle for given side length. In our setting this gives for $e_i \sim e_j$, $e_j \sim e_k$ and $e_k \sim e_i$ ($\sim$ denoting associations):
\begin{align}
\omega (f_3^2)=\sqrt{s (s- \omega(e_i)) \cdot (s-\omega(e_j)) \cdot (s- \omega(e_k))} \; ;\\
s = \frac{\omega(e_i) + \omega(e_j) + \omega(e_k)}{2} \; .
\end{align}

\begin{figure}[h]
	\centering
	\captionsetup{width=0.8\linewidth}
	\includegraphics[width=\linewidth]{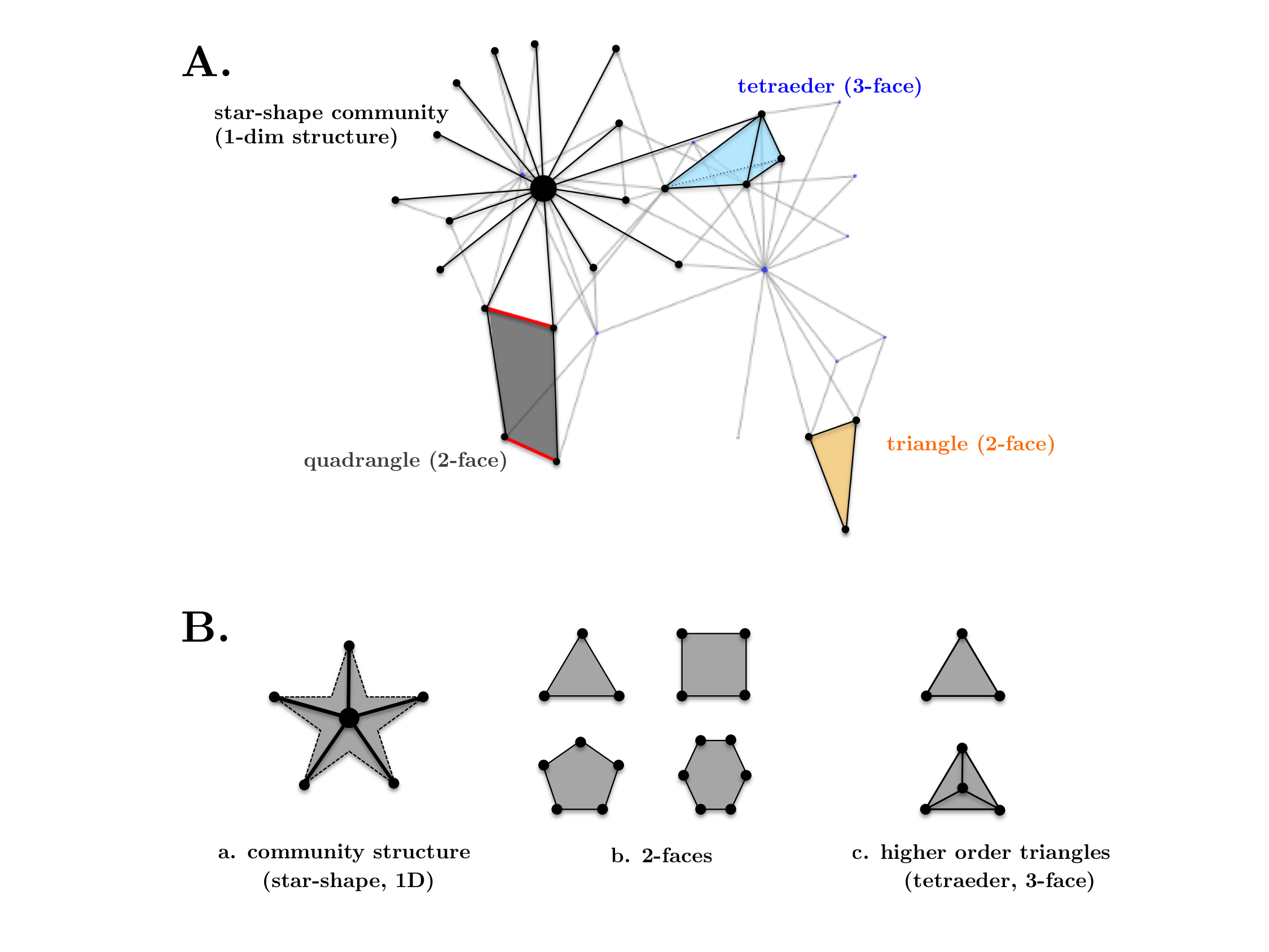}  
	\caption[Extending 1-dimensional graphs to polyhedral complexes]{Extending 1-dimensional graph with higher dimensional faces. \textbf{A:} Complex structures in networks: Community centering around a ``hub'' (black), 2-faces (gray: triangles, orange: quadrangle) and a 3-face (tetrahedron) (network data: \cite{karate}). \textbf{B:} Types of local structures in networks.}
	\label{fig:master}
\end{figure}

\item \textbf{Polygons of higher degree ($d>3$)}
\begin{itemize}
\item \textit{triangulation}:
One could triangulate each $n$-dimensional polygon and use Heron's formula to determine the area of the respective triangles (see above).
\item \textit{imposing coordinates}:
Some real-world networks come naturally with ``coordinates", e.g. information about brain regions in brain networks or locations of resources in energy networks. In those cases, we can calculate the face weights from the edge weights using \textit{Gau{\ss} trapezoid formula}, also known as \textit{Shoelace formula} \cite{shoelace}.
\end{itemize}

\end{enumerate}

\noindent We present statistics on the occurrence of higher degree faces (see Tab. \ref{tab:stats}) to underline their importance in studying the shape of graphs and applications for the analysis of real-world networks, as we will discuss later. Notably, we observe two kinds of behavior: Networks, where we find more faces of higher degree than of lower and vice versa. We will later discuss a possible relation between this observation and the global structure of the networks. The set of real-world networks that we consider here as examplary cases, are displayed in Fig. \ref{fig:examples}.

\begin{table}[H]
\begin{center}   
\begin{small}
\begin{tabular}{ p{0.9cm} p{1.3cm} p{1.3cm} p{1.3cm} p{1.3cm} p{1.3cm} p{1.5cm} p{1.6cm} }
    \hline
    \textbf{data} & \textbf{$\#$nodes}  & \textbf{$\#$edges} & \textbf{$\#$triags.} & \textbf{$\#$quad.} & \textbf{$\#$pent.} & \textbf{$\#$hexag.} & \textbf{avg. deg.} 
    \\ \hline
     (1) & 34 & 78 & 45 & 22 & 5 & 0 & 4.5882 \\
    (2) & 62 & 318 & 95 & 59 & 145 & 239 & 5.1290 \\
    (3) & 1133 & 12035 & 982 & 3434 & 5237 & 8560 & 10.6222\\
  (4) & 79 & 212 & 130 & 17 & 7 & 4 & 5.3711 \\
   \hline 
    \end{tabular}
    \caption[Statistics on higher degree faces]{Higher degree faces in selected real-world networks. We consider examples for social (\textbf{s}, (1) Zachary's karate club \cite{karate} and (2) social interactions of dolphins \cite{dolphins}), peer-to-peer (\textbf{p}, (3) email exchanges \cite{emails1,emails2}) and biological (\textbf{b}, (4) transcription, E. coli \cite{ecoli}) networks. }
    \label{tab:stats}
    \end{small}
    \end{center}
\end{table}

\begin{figure}[t]
	\centering
	\captionsetup{width=0.8\linewidth}
	\includegraphics[width=\linewidth]{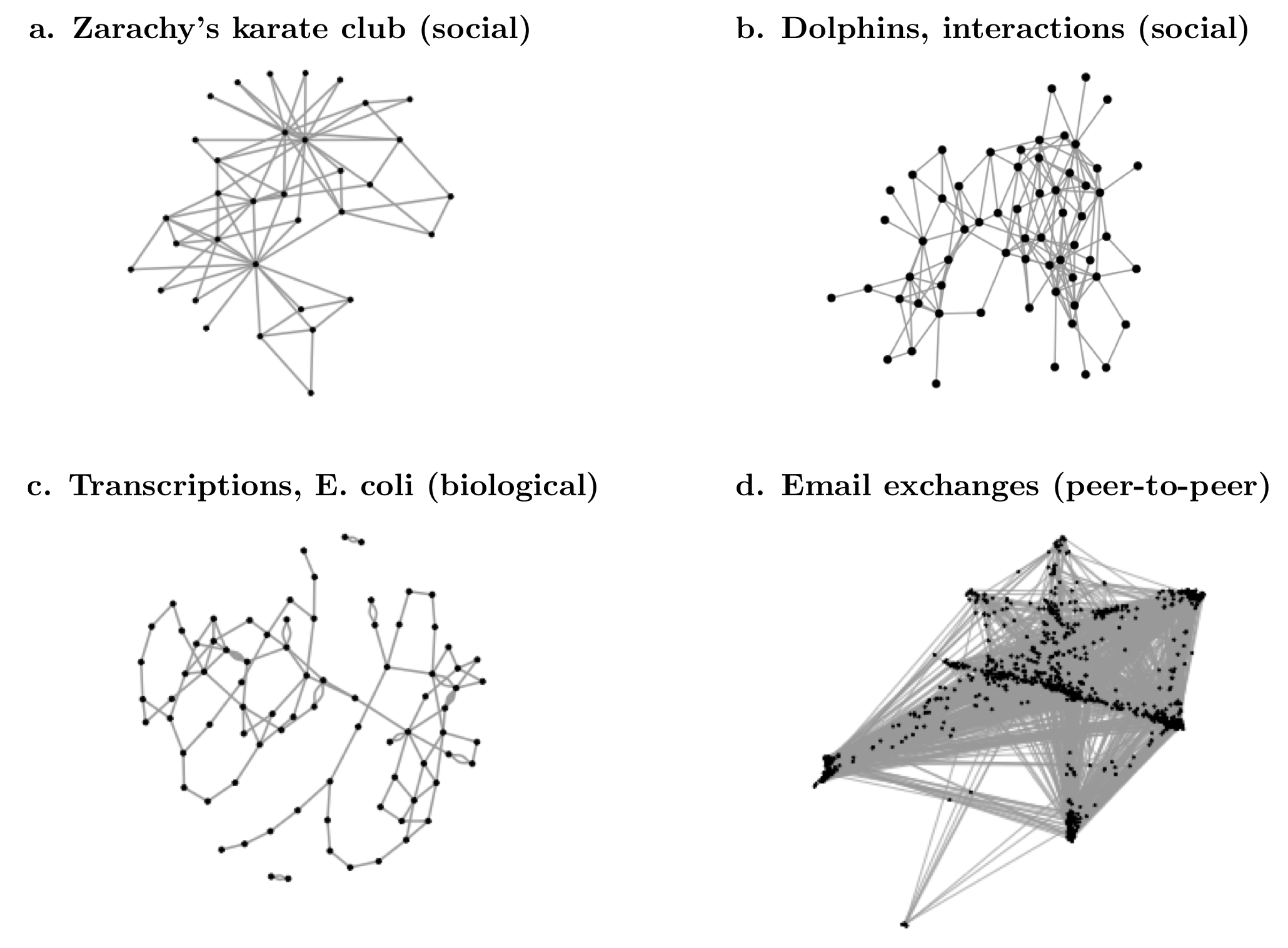}  
	\caption[Examplary real-world networks]{Set of examplary real-world networks that we analyze throughout this article. This includes social (\textbf{s}, (1) Zachary's karate club \cite{karate} and (2) social interactions of dolphins \cite{dolphins}), peer-to-peer (\textbf{p}, (3) email exchanges \cite{emails1,emails2}) and biological (\textbf{b}, (4) transcription, E. coli \cite{ecoli}) networks.}
	\label{fig:examples}
\end{figure}


\subsection{$n$-faces}
\noindent Analogously, one can introduce faces of dimension $n > 2$ by filling in the respective $n$-faces into an $(n-1)$-dimensional polyhedral complex. While this is possible for faces of any degree, we only consider  $n$-dimensional simplices, i.e. triangles for $n=2$, tetrahedra for $n=3$ etc. We impose a weighting scheme using the geometric analogue of the volume of an $n$-dimensional simplex $X$ \cite{stein}:
\begin{eqnarray}
V(X)=\frac{1}{n!} \det \left(	X_1, X_2, ... , X_n	\right) \; ,
\label{eq:simp-vol}
\end{eqnarray}
\noindent where $X_i$ are the geometric representation of the edges. We discuss two possibilities to impose a higher dimensional \textit{weighting scheme} 
\begin{eqnarray}
\omega: F_n(G) \rightarrow \mathbb{R} \; ,
\end{eqnarray}
\noindent based on the edge weights:
\begin{enumerate}
\item \textbf{Unweighted (unit) edges} \\
\begin{minipage}{\linewidth}
\hspace{0.05cm}
\begin{minipage}{0.35\linewidth}
\includegraphics[scale=0.15]{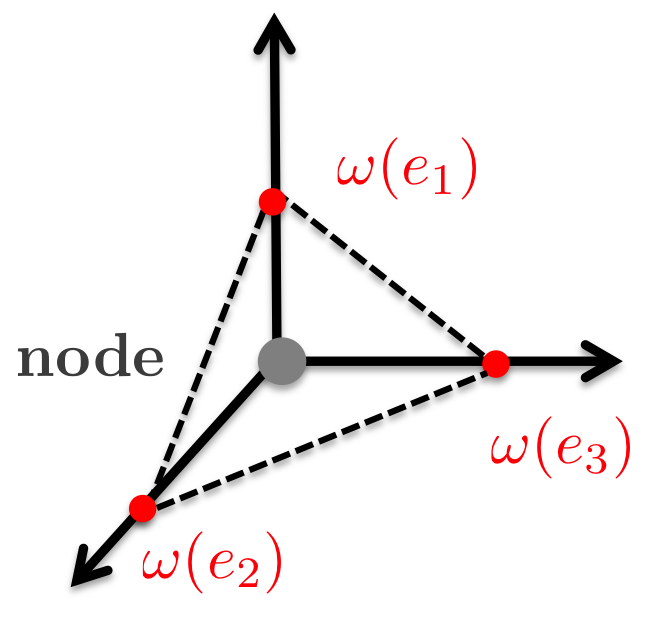} 
\captionof{figure}{Construction \\ of face weights.}
\label{fig:carth}
\end{minipage}
\vspace{0.2cm}
\begin{minipage}{0.45\linewidth}
In the combinatorial case, i.e. if we assume unweighted edges, Eq. (\ref{eq:simp-vol}) simplifies \cite{stein} to 
\begin{align}
V(X) &=\frac{1}{n!} \det \left(	\vec{e}_1, \vec{e}_2, ... , \vec{e}_n	\right) \\
&= \frac{\sqrt{(n+1)}}{n! \cdot \sqrt{2^n}} \; .
\end{align}
From this, we get the following weighting scheme for $n$-faces (of degree 3):
\begin{align}
\omega(f_{n+1} ^n)=\frac{\sqrt{n+1}}{n! \cdot \sqrt{2^n}} \; .
\end{align}
\end{minipage}
\end{minipage}
\item \textbf{Perpendicular triangles} \\
Alternatively, one could map the edge weights to Cartesian coordinates (see Fig.\ref{fig:carth}) yielding \cite{stein}
\begin{align}
X_i &= \omega (e_i) \vec{e}_i \\
\Rightarrow V(X) &=\frac{1}{n!} \det \left(
   \begin{array}{ccc}
     \omega (e_1) \vec{e}_1 & \cdots &  0 \\
     \vdots & \ddots & \vdots \\
     0 & \cdots & \omega (e_n) \vec{e}_n
   \end{array}
\right) = \frac{1}{n!} \Pi_{i=1}^n \omega (e_i) \; ,
\end{align}
and therefore the weighting scheme 
\begin{eqnarray}
\omega(f_{n+1}^n)=\frac{1}{n!} \Pi_{i=1}^n \omega (e_i) \; .
\end{eqnarray}
\end{enumerate}
%

In Tab. \ref{tab:tetra}, we detect higher dimensional faces ($n$-simplices) corresponding to higher order correlations in the already previously considered real-world networks. The results are in accordance to the intuitive expectation that $n$-faces become increasingly rare with growing $n$.
\begin{table}[H]
\begin{center}   
\begin{small}
\begin{tabular}{ p{1cm} p{1.5cm} p{1.5cm} p{1.5cm}  p{1.5cm}  p{1.6cm}  p{1.6cm} p{1.6cm} }
    \hline
    \textbf{data} & \textbf{$\#$ nodes}  & \textbf{$\#$ edges} & \textbf{$\#$ 2-face} & \textbf{$\#$ 3-face} & \textbf{$\#$ 4-faces} 
    \\ \hline
     (1) & 34 & 78 & 45 & 11 & 2  \\
    (2) & 62 & 318 & 95 & 27 & 3  \\
    (3) & 1133 & 12035 & 912 & 745 & 374 \\
   (4) & 79 & 212 & 130 & 38 & 3 \\
   \hline 
    \end{tabular}
    \end{small}
    \caption[Higher dimensional simplicial faces in real-world networks]{Higher dimensional faces ($n$-simplices) in selected real-world networks. We consider examples for social (\textbf{s}, (1) Zachary's karate club \cite{karate} and (2) social interactions of dolphins \cite{dolphins}), peer-to-peer (\textbf{p}, (3) email exchanges \cite{emails1,emails2}) and biological (\textbf{b}, (4) transcription, E. coli \cite{ecoli}) networks.}
    \label{tab:tetra}
    \end{center}
\end{table}


\section{Forman-Ricci curvature in 2D}
%
%
%

\noindent In \cite{WJS}, we introduce Forman-Ricci curvature and its associated flows, namely the Ricci-flow and the Laplace-Bertrami flow for one-dimensional, weighted cell complexes with special emphasis on real-world networks as widely used in the social and biological sciences. Now, we want to map this formalism to higher dimensions by defining Forman's curvature and Laplacian for higher dimensional faces. 

\subsection{Forman-curvature}

Recall Forman's Ricci curvature for regular CW-complexes \cite{Fo}

\[
\hspace*{-0.2cm}
\mathcal{F}(\alpha^p) = w(\alpha^p)\Big[\Big(\sum_{\beta^{p+1}>\alpha^p}\frac{w(\alpha^p)}{w(\beta^{p+1})}\;
+ \sum_{\gamma^{p-1}<\alpha^p}\frac{w(\gamma^{p-1})}{w(\alpha^p)}\Big)\; - \]
\[
\hspace*{0.5cm}
-\sum_{\alpha_1^p\parallel \alpha^p, \alpha_1^p \neq \alpha^p}\Big|\sum_{\substack{\beta^{p+1}>\alpha_1^p \\ \beta^{p+1}>\alpha^p}}\frac{\sqrt{w(\alpha^p)w(\alpha_1^p)}}{w(\beta^{p+1})}\: -
\]
\begin{equation} \label{eq:Forman}
\hspace*{-1.5cm}
- \sum_{\substack{\gamma^{p-1}<\alpha_1^p \\ \gamma^{p-1}<\alpha^p}}\frac{w(\gamma^{p-1})}{\sqrt{w(\alpha^p)w(\alpha_1^p)}}\Big|\:\;\Big] \; ,
\end{equation}

\noindent where $\alpha < \beta$ denotes $\alpha$ being a face of $\beta$ and $\alpha_1 \parallel \alpha_2$ parallel faces $\alpha_1$ and $\alpha_2$. In the case of 2-dimensional polyhedral complexes with faces of up to dimension two, Eq. \ref{eq:Forman} simplifies to the following form:

\begin{defn}[Forman-Ricci curvature, 2-Complex]
\begin{align}
	{\rm Ric}_{\rm F} (e) = \mathcal{F} (\alpha^2) = \omega (e) \left[ \left( \sum_{e \sim f^2} \frac{\omega(e)}{\omega (f^2)}+\sum_{v 			\sim e} \frac{\omega (v)}{\omega (e)}	\right) \right. \nonumber \\
  - \left. \sum_{\hat{e} \parallel e} \left| \sum_{\hat{e},e \sim f^2} \frac{\sqrt{\omega (e) \cdot \omega (\hat{e})}}{\omega (f^2)} - \sum_{v 	\sim e, v \sim \hat{e}} \frac{\omega (v)}{\sqrt{\omega(e) \cdot \omega(\hat{e})}} \right| \right] \; 
 \label{eq:Forman-2d}
\end{align}

\end{defn}

\noindent In the special case of unweighted networks, i.e. $\omega (e) = \omega (v) = 1, \; \forall e \in E(G), v \in V(G)$, the terms simplify to merely the counting of adjacent parents and children, respectively. We get

\begin{align}
{\rm Ric_F} (e) &= \# \lbrace f^2 > e\rbrace + \underbrace{\# \lbrace v < e \rbrace}_{=2} - \left(	\# \lbrace	\hat{e} | \hat{e} \| e \rbrace - \underbrace{\# \lbrace v | v \sim e, v \sim \hat{e}\;, e || \hat{e} 	\rbrace}_{=0}	\right) \\
&= \# \lbrace f^2 > e\rbrace - \# \lbrace	\hat{e} | \hat{e} \| e \rbrace + 2 \; .
\label{eq:Forman-2d-uw}
\end{align}
(The last term in Formula (14) above equals 0 due to the fact that we consider only triangular faces, therefore no parallel edges to $e$ exist, thence the set under consideration is empty.)

\noindent An analogous formula of the degenerated $1$-dimensional case (i.e. with no faces of dimension $>1$) follows immediately from $\# \lbrace \hat{e} | \hat{e} \| e \rbrace = {\rm deg} (v_1 \sim e) + {\rm deg} (v_2 \sim e)$ as

\begin{align}
{\rm Ric_F} (e) = 2 - \sum_{v \sim e} \deg(v) \; .
\end{align}


\noindent An interesting observation when constructing a $2$-dimensional complex from a $1$-dimensional graph can be made in the unweighted case: Adding a face of degree 3 increases the curvature of adjacent edges by 3, a face of degree 4 yields an increment of 2 etc. For faces of degree $n$ we formulate the following lemma:

\begin{lem}
We consider a regular, unweighted polyhedral complex $X$ with $p=2$. When adding a face $f_d^2$ of degree $d$, the Forman-Ricci curvature of the adjacent edges changes to
\begin{eqnarray}
{\rm Ric_F} (e | X+f_d^2) = {\rm Ric_F} (e | X) - d + 6 \; .
\end{eqnarray}
\label{lem:ric-face}
\end{lem}

\noindent \textbf{\textit{Proof}}. The Forman-Ricci curvature for any edge in an unweighted polyhedral complex $X$ is given by Eq. (\ref{eq:Forman-2d-uw}) as
\begin{align*}
{\rm Ric_F} (e | X) = \# \lbrace f_d^2 > e\rbrace - \# \lbrace	\hat{e} | \hat{e} \| e \rbrace + 2 \; .
\end{align*}
Adding a face $f_d^2$ adjacent to $e$ trivially adds 1 to the first term. Furthermore it makes $d-1$ edges potentially parallel to $e$, since they are now sharing a common parent. Among them are two to $e$ adjacent edges that must not be counted twice, leaving us with $d-3$ new parallel edges. However, the two adjacent edges are now having both a common parent and a common child with $e$, i.e. are not parallels anymore. In summary, this gives
\begin{align*}
{\rm Ric_F} (e | X+f_d^2) &= \left( \# \lbrace f_d^2 > e\rbrace +1 \right) - \left( \# \lbrace \hat{e} | \hat{e} \| e \rbrace - 2 + (d-3)\right) +2 \\
&= \# \lbrace f_d^2 > e\rbrace - \# \lbrace \hat{e} | \hat{e} \| e \rbrace -  d +6 \; ,
\end{align*}
or
\begin{align*}
{\rm Ric_F} (e | X+f_d^2) = {\rm Ric_F} (e | X) - d + 6 \; .
\end{align*}
\begin{flushright}
$\Box$
\end{flushright}
\noindent This lemma provides an easy computable formula to study the Ricci curvature across large complex networks with consideration of higher degree faces. It is applicable for cases, where the respective networks come unweighted, or can be mapped to an unweighted one, by imposing a meaningful threshold (regarding the underlying data).

\subsection{Bochner Laplacian}

Analogously, we can define the \textit{ rough Laplacian} via the Bochner-Weizenb{\"o}ck-formula

\begin{eqnarray}
\Delta_{F}^1 X = (\Box_1 - {\bf {\rm Ric_F}})X\ ,
\label{eq:BW}
\end{eqnarray}

\noindent closely related to the Forman-Ricci curvature. In fact, this relation gives rise to an important type of geometric flow, the \textit{Laplacian flow} that has potential applications in the denoising of networks, as we discuss in the previous article \cite{WJS2} for the 1-dimensional case. Recall the general form of the Bochner Laplacian for $p$-dimensional cellular complexes \cite{Fo,WJS}

\begin{defn}[Bochner Laplacian]
	
\[\hspace*{-0.2cm}
\Box_p(\alpha_1^p,\alpha_2^p) =
\sum_{\substack{\beta^{p+1}>\alpha_1^p \\ \beta^{p+1}>\alpha_2^p}}\epsilon_{\alpha_1,\alpha_2,\beta}\frac{\sqrt{w(\alpha_1^p)w(\alpha_2^p)}}{w(\beta^{p+1})}\: +
\]
\begin{equation}  \label{eq:bochner}
\hspace*{1.3cm}
+  \sum_{\substack{\gamma^{p-1}<\alpha_1^p \\ \gamma^{p-1}<\alpha_2^p}}\epsilon_{\alpha_1,\alpha_2,\gamma}\frac{w(\gamma^{p-1})}{\sqrt{w(\alpha_1^p)w(\alpha_2^p)}}\,,
\end{equation}
for which we can write in the special case of 2-dimensional polyhedral complexes:

\begin{align}
\Box_2= \Box_2 (e_1,e_2) = \sum_{e_1 \in f_n^2, e_2 \in f_n^2} \epsilon_{e_1, e_2, f_n^2} \frac{\sqrt{\omega(e_1) \omega(e_2)}}{\sqrt{\omega(f_n^2)}} + \sum_{v \sim e_1, v \sim e_2} \epsilon_{e_1, e_2, v} \frac{\omega(v)}{\sqrt{\omega(e_1) \omega(e_2)}} \; .
\label{eq:bochner-2d}
\end{align}
\end{defn}

\noindent In the case of unweighted networks, this simplifies to the trivial form

\begin{eqnarray}
\Box_2 (e) = \# \; \lbrace f_n^2 > e \rbrace +2 \; .
\end{eqnarray}


%
%
%
%
%


\section{Network-theoretic Gau{\ss}-Bonnet Theorem and Euler characteristic}


Forman's discretization of Ricci curvature satisfies a number of essential properties mandatory for any synthetic definition of Ricci curvature, including discrete versions of \textit{Myers' theorem} on the finiteness of the diameter for spaces satisfying ${\rm Ric} > 0$ \cite{Be} and of \textit{Bochner's theorem} relating the non-negativity of Ricci curvature to the first homology group (see, e.g. \cite{KMM}). However, it fails to satisfy a discrete version of one of the most important classical results, namely the {\em Gau{\ss}-Bonnet Theorem} for 2-dimensional cell complexes. Recall that the classical Gau{\ss}-Bonnet Theorem (for surfaces) states that 
\begin{eqnarray}
\int_SKdA = 2\pi\chi(S)
\end{eqnarray}
represents the {\em Euler characteristic} of the surface $S$. With the {\em Euler characteristic} one can topologically classify manifolds, making a Gau{\ss}-Bonnet-type theorem a much-desired tool for our proposed goal of classifying the shape of networks. 

The failure to construct a Gau{\ss}-Bonnet Theorem in the case of 2-dimensional cell complexes results from the existence of ``metrics'' of negative Forman-Ricci curvature (for every edge $e$) on any simplicial complex of dimension $\geq 2$ (following from results of Gao, Gao-Yau and Lohkamp, see \cite{Fo}). In more detail, now fitting discrete analogue of the Gau{\ss}-Bonnet theorem does not hold for the Ricci curvature in dimension 2, since it implies that both the sphere and the torus can have negative curvature everywhere, which clearly contradicts the Gau{\ss}-Bonnet Theorem. 

The absence of a fitting analogue of the Gau{\ss}-Bonnet theorem represents a serious disadvantage that diminish the elegance of the other results. However, rather recently, an extension of Forman's Ricci curvature was introduced by E. Bloch \cite{Bloch} that does provide a Gau{\ss}-Bonnet-type theorem for cellular complexes. We have recently introduced Forman's Ricci curvature as a characteristic for networks viewed as classic graphs in one dimension and want to build on Bloch's theoretical results for introducing a characterization of the shape of networks, using this formalism.

\subsection{Bloch's combinatorial Forman-Ricci curvature}

Before introducing Bloch's definition, let us note that we consider only the combinatorial case here, i.e. weights $\equiv 1$, reducing the Forman-Ricci curvature to the following simple expression:

\begin{equation} \label{eq:CombRicci}
{\rm Ric_F}(e) = \# \lbrace f^2 > e \rbrace - \# \lbrace \hat{e} : \hat{e} \| e \rbrace + 2\,;
\end{equation}
%


\begin{rem}
	It should be emphasized that, given the fact that the Ricci curvature is a measure (quantity) associated to edges,  i.e. to the common boundary of a number of 2-dimensional faces, only such faces appear in the computation of Forman-Ricci curvature. It follows that (\ref{eq:CombRicci}) is the pertinent one, regardless of the dimensionality of the complex (given the assumption that this dimension is $\geq 2$).
\end{rem}

\noindent Before being able to present Bloch's definition, we first have to introduce some notation (in which we follow Bloch):

\noindent In the following, $F^k=\lbrace f^k\,|\, k = 1,\ldots,j_k	\rbrace$ denotes the set of $k$-dimensional faces of a given cell complex X, and analogous to the former notion, $\# F^k$ the number of faces in the set. 

The {\it Euler characteristic} of a cell complex $X$ is then defined as: 
	
	\begin{equation}    \label{eq:EulerCharacteristic}
	\chi(X) = \sum_{k=0}^{d}(-1)^k F^k\,;
	\end{equation}
where $d$ denotes the dimension of the complex. Furthermore we introduce the following quantities:

\begin{defn}[Auxiliary functions]
Let x be an $i$-dimensional face of a 2-dimensional complex X.
\begin{enumerate} 
\item $A_i(x) = \# \{y \in F_{i+1}, x < y\}$\,;
		
\item $B_i(x) = \# \{z \in F_{i-1}, z < x\}$\,;
		
\item $U_i(x) = \sum_{y > x} B_{i+1}(y)$\,;
		
\item $D_i(x) = \sum_{z < x}A_{i-1}(z)$\,;
		
\item $N_i(x) = N_i^1(x) \Delta N_i^2(x)$, with \\
$	N_i^1(x) = \# \lbrace w \in F_i, \exists v \in F_{i+1} \; {\rm s.t.} \; x < v, w < v \rbrace $, \\
$	N_i^2(x) = \# \lbrace w \in F_i, \exists u \in F_{i-1} \; {\rm s.t.} \; u < x, u < w \rbrace $.

\end{enumerate}

\noindent Here $\Delta$ denotes the symmetric difference and, by convention, summation over the empty set is considered to be zero.
		
\end{defn}

\begin{figure}[h]
	\centering
	\captionsetup{width=0.8\linewidth}
	\includegraphics[width=\linewidth]{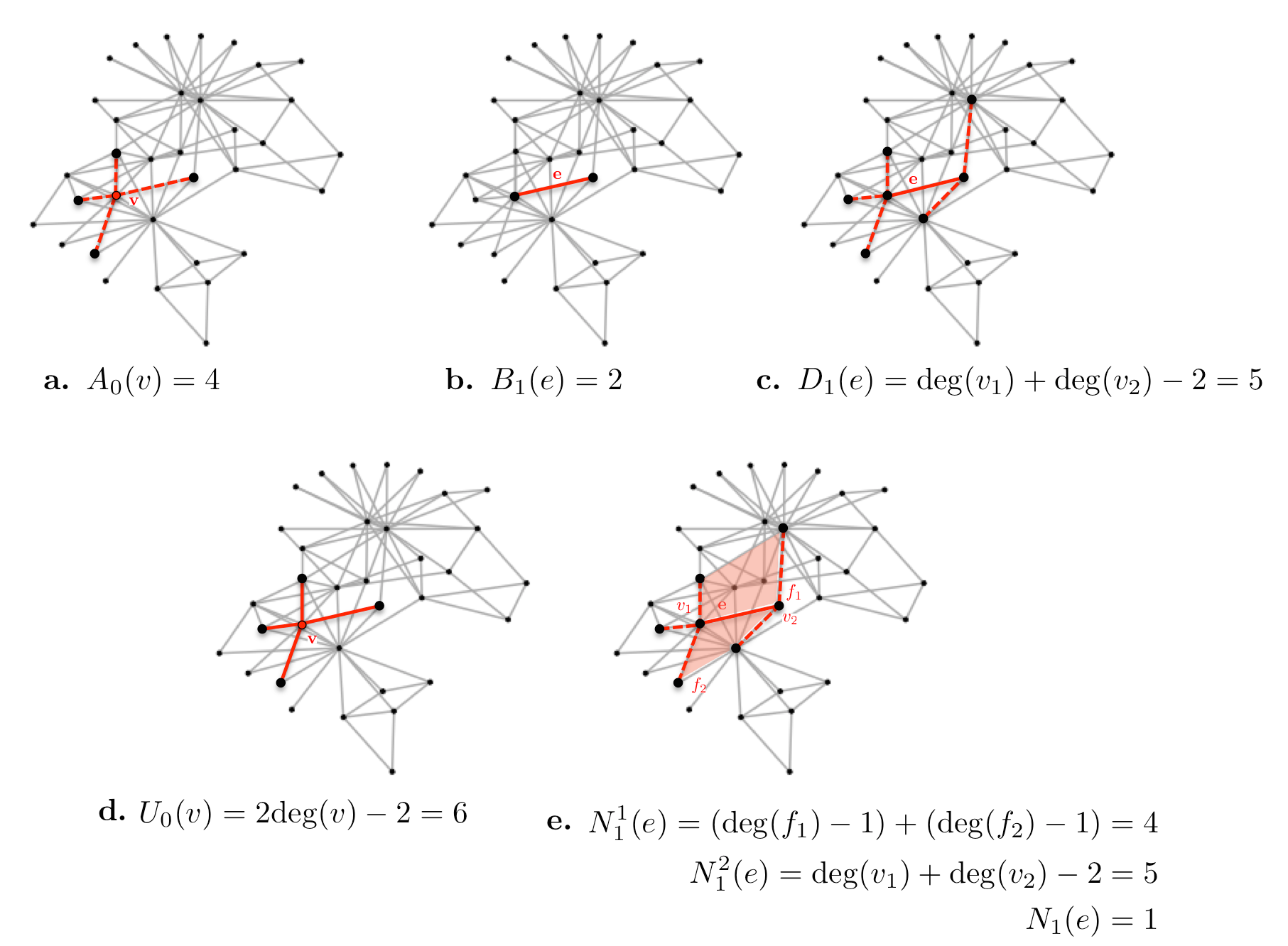}  
	\caption[]{The auxiliary functions (illustrated on Zachary's karate club).}
	\label{fig:examples}
\end{figure}

%
%
%
%
%
%
%
%
%
%
%
	
\noindent The notation above is somewhat technical, however the importance of the newly introduced quantities becomes clearer with the observation that Forman's Ricci curvature can be expressed in terms of these very symbols. More precisely, we have
	
	\begin{equation}
	{\rm Ric_F}(e) = A_1(e) + B_1(e) - N_1(e)\; ,
	\end{equation}
\noindent for all $e \in F_1$. Moreover, using the quantities above, Bloch defines a triplet of discrete \textit{curvature functions}:
	
	\begin{defn}[Curvature-functions]
	\label{def:curv-func}
		Let $X$ be a 2-dimensional cell complex. We define the curvature functions $R_i:F_i \rightarrow \mathbb{R}$, $i = 1,2,3$ as follows:
	\begin{enumerate}
	\item $R_0 (v) = 1 + \frac{3}{2}A_0(v) - A_0^2(v)\,;$
	\item $R_1(e) = 1 + 6A_1(e) + \frac{3}{2}B_1(e) - U_1(e) - D_1(e)\,;$
	\item $R_2(f) = 1 + 6B_2(f) - B_2^2(f)\,;$
%
%
%
%
%
		%
		\end{enumerate}
		where $v,e, f$ denote vertices, edges and faces, respectively, i.e. $v \in F_0, e \in F_1, f \in F_2$.
\end{defn}
	
	
\subsection{Network-theoretic Gau{\ss}-Bonnet theorem}
	
	We can now formulate Bloch's -- from our viewpoint -- most important result:
	
	\begin{thm}[Bloch, \cite{Bloch}, Theorem 2.4] \label{thm:Bloch}
		Let $X$ be 2-dimensional cell complex. Then:
		
		\begin{equation}
		\sum_{v \in F_0}R_0(v) - \sum_{e \in F_1}R_1(e) + \sum_{f \in F_2}R_2(f) = \chi(X)\,.
		\end{equation}
	\end{thm}
	
\noindent The relevance of this Gau{\ss}-Bonnet type theorem follows from the following important identity between the 1-dimensional curvature functions and the Ricci curvature for polyhedral complexes:
\begin{equation} \label{eq:ER1=Ric}
		R_1(e) = {\rm Ric}(e)\,,
\end{equation} 
\noindent for all $e \in F_1$. Since polyhedral complexes represent a special case of cell complexes, the last two formulas above imply that, for a 2-dimensional polyhedral complex, the following holds:

\begin{equation}  \label{eq:Bloch-GaussBonnet}
\sum_{v \in F_0}R_0(v) - \sum_{e \in F_1}{\rm Ric}(e) + \sum_{f \in F_2}R_2(f) = \chi(X)\,.
\end{equation}

\begin{rem}
\noindent A sufficient condition for the equality (\ref{eq:ER1=Ric}) is that the intersection of any pair of 2-cells (faces) contains at most one 1-cell, i.e. an edge. Unfortunately, in our context this condition fails in general, because it is possible that, for instance, two edges adjacent to a vertex are part of both a triangle and a quadrangle, such that the triangle does not represent a subset of the quadrangle (i.e. they describe different types and orders of correlation).\\
\end{rem}	
	
	

%
\begin{rem}
In our context, the importance of Theorem \ref{thm:Bloch} above resides in the fact that it allows us to define, in strict analogy to the surface (both smooth and discrete) case, a {\em prototype} or (or {\em reference}) {\em network} for the Forman-Ricci flow in terms of discrete curvatures, and in particular of the extended Forman-Ricci curvature in a mathematically sound manner. By ``pruning'' the leaves, i.e. neglecting isolated edges (not at the boundary of any 2-face), one can construct the desired 2-dimensional prototype \footnote{The reason we prefer the less common name of ``prototype'', rather than ``model'' networks, is that the term ``model networks'' is already established in the field, and denotes such combinatorial/probabilistic constructs as those proposed by Erd\"{o}s and Renyi \cite{er1,er2}, Watts and Strogatz \cite{ws} and Albert and Barabasi  \cite{ab}.} networks. \\
\end{rem}

\noindent For the specific case of 2-dimensional complexes that we consider here, the curvature functions defined in Def. (\ref{def:curv-func}) simplify to
\begin{align}
R_0 (v) &= 1 + \frac{3}{2} {\rm deg}(v) - {\rm deg}^2 (v) \; ;\\
R_1 (e) &= 4 + 6 \cdot \# \lbrace f_n^2	 > e \rbrace - \sum_{f_n^2 > e} n - \sum_{v < e} {\rm deg} (v) \; ;\\
R_2 (f_n^2) &= 1 + 6n + n^2 \; .
\label{eq:curv-func}
\end{align}
With this, we get a directly computable form of the Euler-characteristic:
\begin{defn}[Euler characteristic for 2-dimensional polyhedral complexes]
\begin{align}
\chi (X) &= \sum_{v \in V(X)} \left( 1 + \frac{3}{2} {\rm deg}(v) - 	{\rm deg}^2(v)	\right)  \\
&- \sum_{e \in E(X)} \left(	4 + 6 \cdot \# \lbrace	f_n^2 > e	\rbrace	 + \sum_{f_n^2 > e} n + \sum_{v < e} {\rm deg}(v) \right) \\
&+ \sum_{f_n^2 \in F_n(X)} 1 + 6n + n^2 \; .
\end{align}


\end{defn}

%
	
%
%

\subsection{A topological implication}

One of the main strengths of Forman's curvature resides in the fact that its sign provides important information on the topology of the underlying complex. While not all of Forman's results can be generalized to the case of Bloch's extension of the original definition, some topological implications can be transferred as discussed below. However, this is only possible in terms of the mean of the curvature function $R_1$ rather than its edge-valued, basic definition. This is a consequence of the fact that in the Gau{\ss}-Bonnet Theorem (\ref{eq:Bloch-GaussBonnet}) above, $R_1$ appears with a ``-'' sign. This fact also explains why we cannot expect to obtain a result of the required type in the most general case, but rather have to consider some (combinatorial) restrictions. A number of such restrictions \cite{Bloch} are listed below, in increasing order of generality:

\begin{enumerate}
	\item $\overline{B}_1 \geq \frac{20}{9}$, where $\overline{B}_1 = \frac{1}{\# F_1}\sum_{e \in F_1}B_1(e)$;
	
	\item $\overline{B}_1 = 2$ and  $\overline{A}_1 \geq 2$, where $\overline{A}_1 = \frac{1}{\# F_1}\sum_{e \in F_1}A_1(e)$;
	
	\item \((\overline{A}_1 + \overline{B}_1)^2 - 6\overline{A}_1 -  \frac{3}{2}\overline{B}_1 - 1 \geq 0\).
\end{enumerate}

\noindent The first -- and main -- result regarding a topological characterization of a complex in terms of the mean value of $R_1$, namely $\overline{R}_1 = \frac{1}{F_1}\sum_{e \in F_1}R_1(e)$, is the following:

\begin{thm}[Bloch, \cite{Bloch}, Theorem 2.7] \label{thm:Bloch1}
	Let $X$ be a 2-dimensional cell complex, satisfying any one of the conditions (1) - (3) above. Then, if $\overline{R}_1 >0$, then $\chi(X) > 0$.
\end{thm}

\noindent For the here considered special case of 2-dimensional, polyhedral complexes, we obtain the following lemma:
\begin{lem}
\label{lem:chi-curv}
Let X be a 2-dimensional polyhedral complex with (1) $\overline{A}_1 \geq 2$ and (2) $\overline{R}_1 > 0$, then $\chi(X)>0$.
\end{lem}
\begin{proof}
We show that for a 2-dimensional polyhedral complex, conditions (2) and (3) in Thm. \ref{thm:Bloch1} imply $\overline{A}_1 \geq 2$. For this, we first compute $\overline{B}_1$ and $\overline{A}_1$:
\begin{align}
B_1 (e) = \# \lbrace	v \in F_0, \; v < e \rbrace = 2 \; ; \\
\overline{B}_1 = \frac{1}{\# F_1} \left( \# F_1 \cdot 2 \right) = 2 \; ;
\end{align}
and
\begin{align}
A_1 (e) = \# \lbrace	f \in F_2, \; e < f	\rbrace = \# \lbrace	f>e \rbrace \; ; \\
\overline{A}_1 = \frac{1}{\# F_2} \sum_{e \in F_1} \# \lbrace	f>e	\rbrace \; .
\end{align}
\noindent From this, we see that $\overline{B}_1=2$ holds trivially, i.e. (2) simplifies to $\overline{A}_1 \geq 2$.
\begin{rem}
Since $B_1 (e)=2 \quad \forall e \in F_1$, condition (1) is never fulfilled.
\end{rem}
\noindent Furthermore, condition (3) gives
\begin{align}
\left(	 \overline{A}_1 + \overline{B}_1 \right)^2 - 6 \overline{A}_1 - \frac{3}{2} \overline{B}_1 - 1 \geq 0 \\
^{\overline{B_1}=2}\Rightarrow \overline{A}_1 ^2 - 2 \overline{A}_1 \geq 0 \; ,
\end{align}
i.e. $\overline{A}_1 \geq 2$.
\end{proof}

\begin{rem}
The condition given by Lemma \ref{lem:chi-curv} link back to our earlier comment on the choice of polyhedral complexes as higher dimensional representations for networks. This choice imposes restrictions on both the connectivity and possibly occurring degenerated substructures. Lemma\ref{lem:chi-curv} only holds for complexes without isolated edges, i.e. is not generally applicable to real-world networks. However, by ``pruning the leaves", i.e. by neglecting isolated nodes and edges one could overcome this issue. 
\end{rem}


\noindent We compute\footnote{See \textit{Supplemental Material} for details on the implementation and where to find the publicly available code.} $\chi (X)$ for the previously considered real-world networks (see Tab. \ref{tab:eu-char}). To limit computational expense, we only consider 2-faces of degree 3, i.e. triangles. An interesting investigation extending the present study would be to include quadrangles, pentagons etc. and to analyze how the Euler characteristic changes for different types of networks. Note, however, that filling in large numbers of higher faces of dimension 2 or higher would lead to a simply connected complex, i.e. would represent a degenerate case of the formalism introduced in this article. Observe also that, due to the large number of triangles, the Euler characteristic even of medium-sized networks is huge. Given this fact, and taking into account that only the sign of $\chi$ is truly relevant for our analysis (see discussion below), we compute and tabulate, instead, the {\em mean} Euler characteristic $\overline{\chi} = \lfloor\overline{\chi}/\#T\rfloor$, where $\#T$ denotes the number of triangular faces.

An evaluation of the computed $\chi (X)$ with the results from Tab.\ref{tab:stats} suggests that networks with a high number of high-degree faces have positive Euler characteristics. On the contrary, low numbers of high-degree faces might hint on negative Euler characteristics.


\begin{table}[H]
\begin{center}   
\begin{tabular}{ p{1cm}  p{2cm}  p{2cm} p{2cm}  p{2cm}}
    \hline
    \textbf{data} & \textbf{$\#$ nodes}  & \textbf{$\#$ edges} & \textbf{$\#$ 2-face} & \textbf{$\overline{\chi}(X)$} 
    \\ \hline
     (1) & 34 & 78 & 45 & 1$^*$ \\
    (2) & 62 & 318 & 95 & ${\rm 20 > 0}$ \\
    (3) & 1133 & 12035 & 982 & ${\rm 15 > 0}$ \\
    (4) & 79 & 212 & 130 & ${\rm -1 < 0}$ \\
   \hline 
    \end{tabular}
    \caption[The mean Euler characteristic for real-world networks]{The mean Euler characteristic $\overline{\chi}$ for selected real-world networks. We consider examples for social (\textbf{s}, (a) Zachary's karate club \cite{karate} and (b) social interactions of dolphins \cite{dolphins}), peer-to-peer (\textbf{p}, (d) email exchanges \cite{emails1,emails2}) and biological (\textbf{b}, (c) transcription, E. coli \cite{ecoli}) networks. For the computation we used normalized curvature functions. $^*$: As discussed earlier, our formalism excludes degenerated substructures like isolated edges. The small, but positive $\chi$ in this network is possibly an artifact introduced by the high number of such degenerated substructures resulting from the very small size of the network.}
    \label{tab:eu-char}
    \end{center}
\end{table}


\section{Ricci flow on 2D complexes}
%
%
%
%
%
%

\noindent We have extended the classical notion of the network to higher dimensions by adding $n$-dimensional faces and defined a Ricci curvature on those structures following previous, more general work of Forman \cite{Fo}. Well appointed with these theoretical tools, we can now define associated curvature flows on the networks 'surface'. For this, we build on previous work by E. Bloch and work on the combinatorial flow by Chow and Luo (see \cite{To}).

We assume here that curvature on the ``surface" will remain finite at all times, thus assuring the convergence to a unique connected reference space (\textit{prototype complex} or \textit{prototype network}). This assumption is motivated by results for the flow on smooth surfaces and the already mentioned combinatorial flow. In the 2-dimensional case, we assume that we can achieve this by cutting isolated edges (i.e. faces without parents), i.e. those that are disconnected from the network's ``surface". Note, that this assumption is not necessarily true for higher dimensional cases, e.g. one cannot hope to observe finite curvature in polyhedral 3-complexes, since for the classical analogue -- the (smooth) flow on 3-manifolds -- such ``blow-ups" of infinite curvature are known to occur \cite{To}. Moreover, while there are algorithmic tools for combinatorial singularities arising in the surface flow, the 3-dimensional case is far more complicated and no analogous method has been introduced so far \cite{saucan1}.

\subsection{Long term flow and prototype networks}
As we already suggested in a previous article \cite{WJS2}, the potentially strongest applications of the Forman-Ricci flow to real-world networks is network extrapolation and prediction via the use of the \textit{long term} version of the  Forman-Ricci flow
\begin{equation} \label{eq:RicciFlowNtwks}
\tilde{\gamma} (e)  - \gamma (e) = - {\rm Ric}_F (\gamma (e)) \cdot \gamma (e)\,,%
\end{equation}
where $\tilde{\gamma} (e)$ denotes the new (updated) weighting scheme $\gamma (e)$ with $\gamma (e)$ the initial (given) one. Since the flow above is modeled after the classical one for surfaces, namely
\begin{align}
{\frac{\partial g_{ij}}{\partial t} = - \overline{K} (g_{ij}) \cdot g_{ij}}\,; \\
\overline{K} = \frac{\int_{S^2}KdA}{\int_{S^2}dA}\,; 
\end{align}
where  $\overline{K}$ denotes the mean Gaussian curvature of the given surface $S^2$ and $dA$ the respective area element.
%
%
One would hope that the long term flow would share some essential properties with its classical counterpart, mainly the evolution -- without the formation of singularities -- to a {\it prototype} or {\it reference space}. In the classic case, both smooth surfaces under the classical Ricci flow \cite{Ha}, and the piecewise-flat ones under the discrete flow \cite{CL}, evolve to model surfaces of constant Gaussian curvature. In consequence, each compact surface (smooth or piecewise-flat) admits a {\it background metric} of constant curvature covered by the 2-sphere, the flat torus and the hyperbolic plane. Moreover, the limit surface (and hence its background metric) is determined by the topology of the surface, more precisely by its {\em Euler characteristic}.

However, in order to obtain this behavior without the metric contracting and the surface becoming infinitesimal in the limit, we use the {\it normalized} flow instead. For surfaces it has the following form:
\begin{equation}  \label{eq:NormalizedRicciFlow}
{\frac{\partial g_{ij}}{\partial t} = - \left(K (g_{ij}) - \overline{K}(g_{ij})\right) \cdot g_{ij}}\,;
\end{equation}
%


\noindent In analogy with (\ref{eq:NormalizedRicciFlow}) we define the {\it Normalized Forman-Ricci flow} for networks as 
\begin{eqnarray} \label{eq:NormalizedFormanRicciFlow}
\frac{\partial \gamma(e)}{\partial t} = -\left( \rm{Ric_F} \left( \gamma(e) \right)  - \rm{\overline{Ric}_F}\right) \cdot \gamma(e) \; ,
\end{eqnarray}
where $\rm{\overline{Ric}_F}$ denotes the mean Forman-Ricci curvature.

However, there is no theoretical proof to support the role of $\rm{\overline{Ric}_F}$ in concordance with that of $\overline{K}$. Moreover there is no proper, established notion of Euler characteristics for an abstract, non-planar graph. Even if one could overcome this problem, one is still confronted with the failure of Forman's curvature to satisfy a Gau{\ss}-Bonnet type theorem and hence the fact that there is no way to associate a limit surface (and a background geometry) to a given network. This shortcoming has largely motivated the present work.

Fortunately, the approached adopted here, based on Bloch's extension of Forman's Ricci curvature has two advantages: It first allows for a meaningful \textit{mean Ricci curvature} and, moreover, satisfies a Gau{\ss}-Bonnet Theorem, thus enabling us to define \textit{prototype} (or \textit{reference}) \textit{networks}, as minimal (in the sense of having the minimal number of 2-faces) 2-dimensional polyhedral complexes with the Euler characteristics prescribed by the Gau{\ss}-Bonnet formula (\ref{eq:Bloch-GaussBonnet}). In consequence, one can define a network to be {\em spherical}, {\em Euclidean} or {\em hyperbolic}, if its Euler characteristic is $>, =$ or $<$ 0, respectively. 
\begin{defn}[Prototype networks]
Let X be a 2-dimensional polyhedral complex with Euler characteristic $\chi$ as given by the Gau{\ss}-Bonnet formula. Then we define $\chi$ to be
\begin{enumerate}
\item Spherical, if $\chi >0$ \; ,
\item Euclidean, if $\chi = 0$ \; ,
\item Hyperbolic, if $\chi <0$ \; .
\end{enumerate}
\end{defn}

\noindent While the minimality condition above is simple and intuitive, it seems that proving convergence of the flow using this definition might be somewhat difficult. Therefore, in analogy with the classical (surfaces) case, we suggest that, for the purpose above, as well for computational efficiency, one should rather define the prototype network as having constant curvature, more precisely such that $R_1 = {\rm const.}$. (The choice of $R_1$ as curvature is natural, given it represents an extended Ricci curvature and considering its essential role in defining $\chi$.) While determining if the supposed identity between the two approaches to define prototype networks represents work in progress, let us only mention here that first experimental results seem to suggest that this is, indeed, the case.

\noindent Thus, by using Bloch's results, we can define a consistent notion of background geometry for networks that allows for the study of long term evolution of networks. Furthermore, it enables us to study their (topological) complexity, dispersion of geodesics, recurrence, volume growth, etc. -- in fact all the essential properties captured by its geometric type.

Surely, there are a number of caveats and technical issues that still need to be overcome. Firstly, the formalism is restricted to the 2-dimensional case, i.e. to the addition of 2-faces. This follows from the fact that with transition to dimensions higher or equal 3, the flow would necessarily develop singularities as discussed above. Secondly, for computational investigations one needs to remove edges and nodes of degree one to account for degenerated substructures that are not covered by the class of polyhedral complexes. For practical purposes, one might be restricted to analyzing the largest connected component when imposing a higher dimensional connectivity condition since 1-connected graphs, etc. would render 2-complexes from having topological (cone) singularities. 

A natural question arising in this context is how the above defined flow compares to similar geometric flows, starting with the simplest -- the combinatorial flow.

For the transition from the Forman-Ricci flow, defined on the edges, to a flow on the 2-dimensional complex we used the fact that the faces are just Euclidean (triangles)\footnote{In fact, this can be done, {\it mutatis mutandis} for, say, spherical triangles as well.}, thus allowing for an extension from edges to faces by considering simple -- and standard -- barycentric combinations.

\subsection{Ricci flow and Laplacian flow}
In the previous section, we mentioned two definitions for the Ricci flow on the surface of 2-dimensional complexes, namely the \textit{short term} (Eq. (\ref{eq:RicciFlowNtwks})) and the \textit{long term Ricci flow} (Eq. (\ref{eq:NormalizedFormanRicciFlow})). Since we want to study the shape of networks in an evolutionary sense, i.e. in terms of the above introduced prototype complexes, we consider the normalized long term Ricci flow. In the case of 2-dimensional polyhedral complexes, this gives with
\begin{align}
{\rm Ric_F}(e)=R_1 (e) \; ; \\
\overline{{\rm Ric_F}} = \frac{1}{\# F_1} \sum_e R_1 (e) \; ,
\end{align}
the following form:
\begin{defn}[Ricci-flow for polyhedral complexes]
Let $X$ be a regular 2-dimensional, polyhedral complex. Then the Ricci-flow on the surface can be defined in terms of the edges as
\begin{eqnarray}
\frac{\partial \gamma (e)}{\partial t} = \left(  R_1 (e) - \frac{1}{\# F_1} \sum_{e' \in F_1} R_1(e') \right) \gamma(e) \; .
\end{eqnarray}
\end{defn}
\noindent Here we consider the \textit{normalized} long term flow scaled with the mean Ricci curvature $\overline{{\rm Ric_F}}$, a global network property that is strongly related to the Euler-characteristic, as we have seen in the previous section.

Combining the Bochner-Weizenb{\"o}ck formula (Eq. \ref{eq:BW}) and the introduced Ricci-flow, we can define a Laplacian flow $\Delta_F ^2$ for 2-dimensional polyhedral complexes. We again only consider the combinatorial, i.e. unweighted, case and set ${\rm Ric_F}=R_1$:
\begin{defn}[Rough Laplacian for polyhedral complexes]
Let $X$ be a regular 2-dimensional, polyhedral complex. Then the {\em rough Laplacian} on the surface can be defined in terms of the edges as (with Eq. (\ref{eq:bochner-2d}) and Eq. (\ref{eq:curv-func})):
\begin{eqnarray}
\Delta_F ^2 = \sum_{f_n^2 >e} n - 5 \cdot \# \lbrace	 f_n^2 > e \rbrace + \sum_{v < e} {\rm deg}(v) - 2 \; .
\end{eqnarray}
\end{defn}

\subsection{Gauge networks}

As previously discussed in this section, we expect the Ricci-flow to drive networks to structural limiting cases that we term \textit{prototype networks}, characterized by the sign of their Euler characteristic. Fig. \ref{fig:gauge} shows the evolution of two real-world examples \cite{karate, dolphins} with the discrete Ricci-flow acting on edge weights. We observe the evolution of the \textit{hyperbolic} case and the \textit{spherical} case as defined above. ()The degenerated \textit{Euclidean} case with $\chi=0$ is a rare limiting case that may only accidentally occur in real-world examples, i.e. it is primarily of theoretical interest for systematic reasons.) While a formal, theoretical proof of convergence to such limit cases eludes us at this point in time, the tentative experiments so far indicate that, indeed, that networks with $\overline{{\rm Ric_F}} > 0$ tend to evolve to a more ``round'', spherical shape, whereas tose with $\overline{{\rm Ric_F}} < 0$ evolve to more tree-like, hence hyperbolic structures.


Our phenomenological exploration of the occurrence of higher dimensional faces in a small set of real-world networks suggested two distinct types of networks: Those that have a higher number of 2-faces for higher degrees and others with a very low number of detected 2-faces of higher degrees. A comparison of this observation and the computed Euler characteristics suggests that the two types represent the two main classes of prototype networks that occur in real-world networks: spherical and hyperbolic. \emph{Spherical} networks seem to be characterized by a high number of high degree faces. They occur in densely connected substructures that possibly govern the evolution of the spherical type as illustrated in Fig. \ref{fig:gauge}B. On the contrary, our \emph{hyperbolic} examples show a rapidly decreasing number of detected faces when considering higher degree faces. This lack of higher degree faces seems to be linked to the rather sparse community structure that eventually evolves to a hyperbolic prototype network (illustrated in Fig. \ref{fig:gauge}B).

\begin{figure}[H]
	\centering
	\captionsetup{width=0.9\linewidth}
	\includegraphics[width=0.75\linewidth]{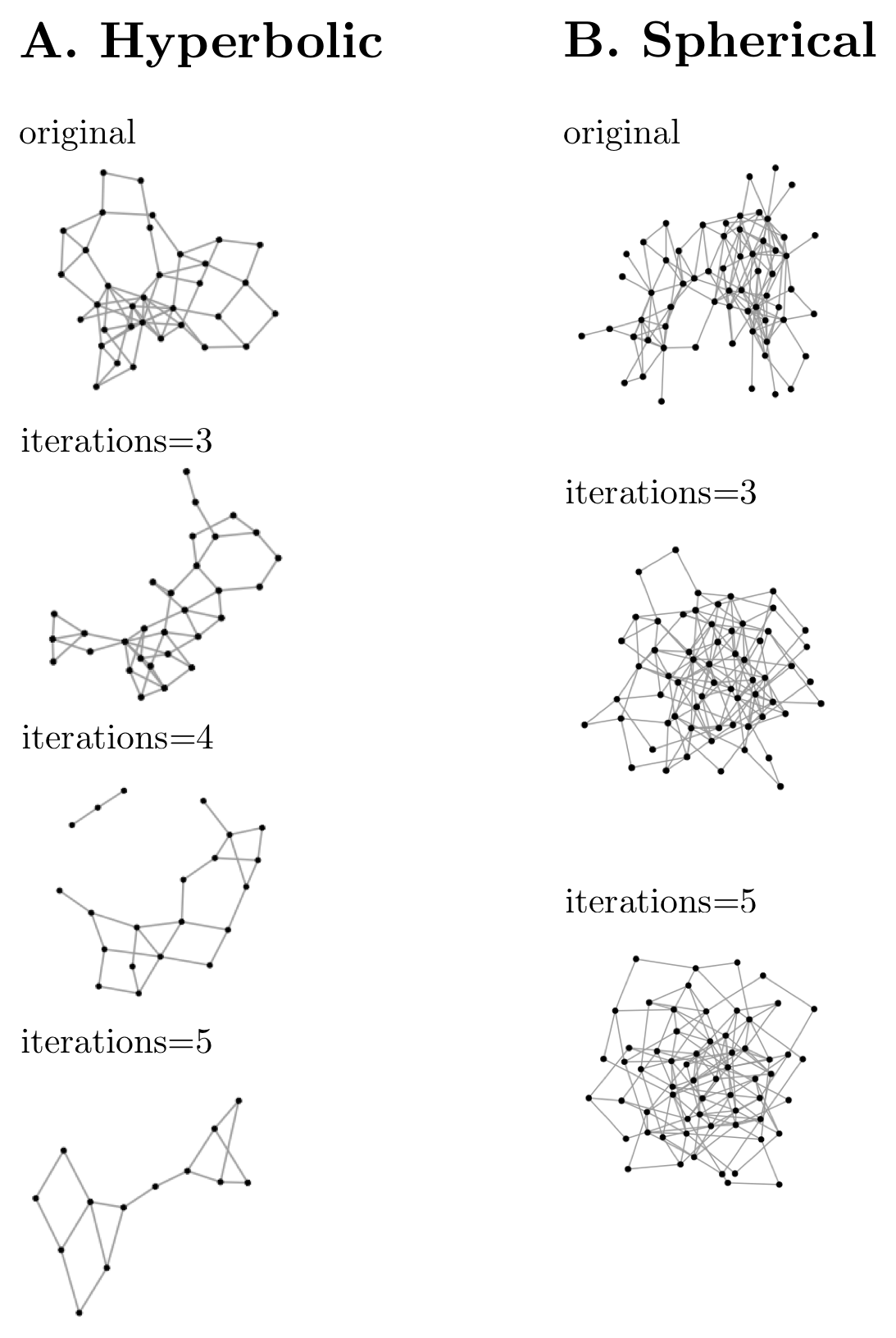} 
	\caption[Prototype networks]{Prototype networks as limiting cases under the Ricci flow acting on edge weights. Edges with normalized weights below a threshold $0.05$ are considered to vanish. \textbf{A:} Hyperbolic case. (network data: \cite{karate}). \textbf{B:} Spherical case. (network data: \cite{dolphins}).}
	\label{fig:gauge}
\end{figure}
\newpage

\section{Implications for real-world networks}

\noindent After considering a number of abstract concepts regarding the shape of networks, we now want to put our results in context with practical applications for real-world networks. 

\subsection{Higher degree faces and the ``backbone" effect}
While ``filling in" higher degree faces marks, on an abstract level, the transition between the classic graphs in one dimension and the higher-dimensional polyhedral complexes, there exists a practical meaning of this transition for real-world networks. One important class of real-world networks are correlation networks, where edges represent correlations between objects inferred from empirical data. The edge weights reflect the strength of the respective correlation. Such networks are widely used in the biological sciences. For instance, co-expression networks represent correlations in expression levels of genes and brain networks are typically inferred from correlations in the activity profiles of brain regions.  

A \textit{2-face} of degree $n$ then represents a sequence of $n$ correlated elements of the system (represented by nodes), e. g. for $n=3$ we have elements $A$, $B$ and $C$ where $A$ is correlated with $B$, $B$ is correlated with $C$ and $C$ is correlated with $A$. By ``filling in" the respective triangle $f_3^2=\lbrace A, B, C \rbrace$ we contract the three elements -- curvature-wise -- to one. An analogous observation can be made for higher degree \textit{2-faces} (see Lemma \ref{lem:ric-face}) and $n$-faces in general, as well as more general for other classes of real-world networks.

The contraction itself is a meaningful reduction for most classes of real-world networks and can be understood as representing the network at a higher level of abstraction: The elements (nodes) jointly represented by a $n$-face share commonalities and can be seen as a group or cluster with respect to this commonality. For example, in a social network, a $n$-simplex could represent a group of friends, co-workers, classmates or collaborators. Instead of representing every individual in the network, we compress them in groups; simplifying the computation of network properties, such as the Ricci curvature, by large scales. This transition to a higher level of abstraction is an intrinsic property of the higher degree Ricci curvature that we introduced in this article. In a previous article \cite{WJS2}, we discussed this very aspect in the 1-dimensional case as the \textit{backbone-effect} of the Ricci curvature, i.e. its ability to highlight essential topological and structural information.

\subsection{Relation to dispersion}
The observation of 2-faces in networks and the importance of this structural information have been made previously. 
We discuss a related approach from Social Network Science, namely the \textit{dispersion} \cite{BK}.

\textit{Dispersion} has been introduced in the context of Social Networks as a characteristic for structural information underlying weak and strong ties. We speak of \textit{weak ties}, in the case of edges that connect nodes that share a low number of common neighbors, i.e. nodes with path distance 1 to both original nodes. In contrast, \textit{strong ties} connect nodes with a high number of common neighbors. The dispersion is then given by
\begin{eqnarray}
{\rm disp} (u,v) = \sum_{i, j \in C_{uv} \setminus \{u,v\}} d_v(i,j) \; ,
\label{eq:disp}
\end{eqnarray}
where 
$C_{uv}$ represents the set of common neighbors of $u$ and $v$, and $d_v$ the distance on $C_{uv}$ that has been shown to perform best \cite{BK}, namely
\begin{eqnarray}
d_v (i,j)=\begin{cases}	1, & i \nsim j \; \mbox{and} \; C_{ij} = \varnothing \\ 0, & {\rm else}\end{cases} \;.
\end{eqnarray}
The edges that contribute to dispersion are parallels of $e=e(u,v)$, i.e. also contribute to the Ricci curvature of $e$. Therefore, the contraction that occurs curvature-wise when filling in faces also reduces dispersion. Closely related to the combinatorial (unweighted) Ricci curvature, the dispersion is an edge-based network property that -- while also dependent on node degrees -- strongly characterizes edge-based information.

\begin{figure}[H]
	\centering
	\captionsetup{width=0.9\linewidth}
	\includegraphics[width=\linewidth]{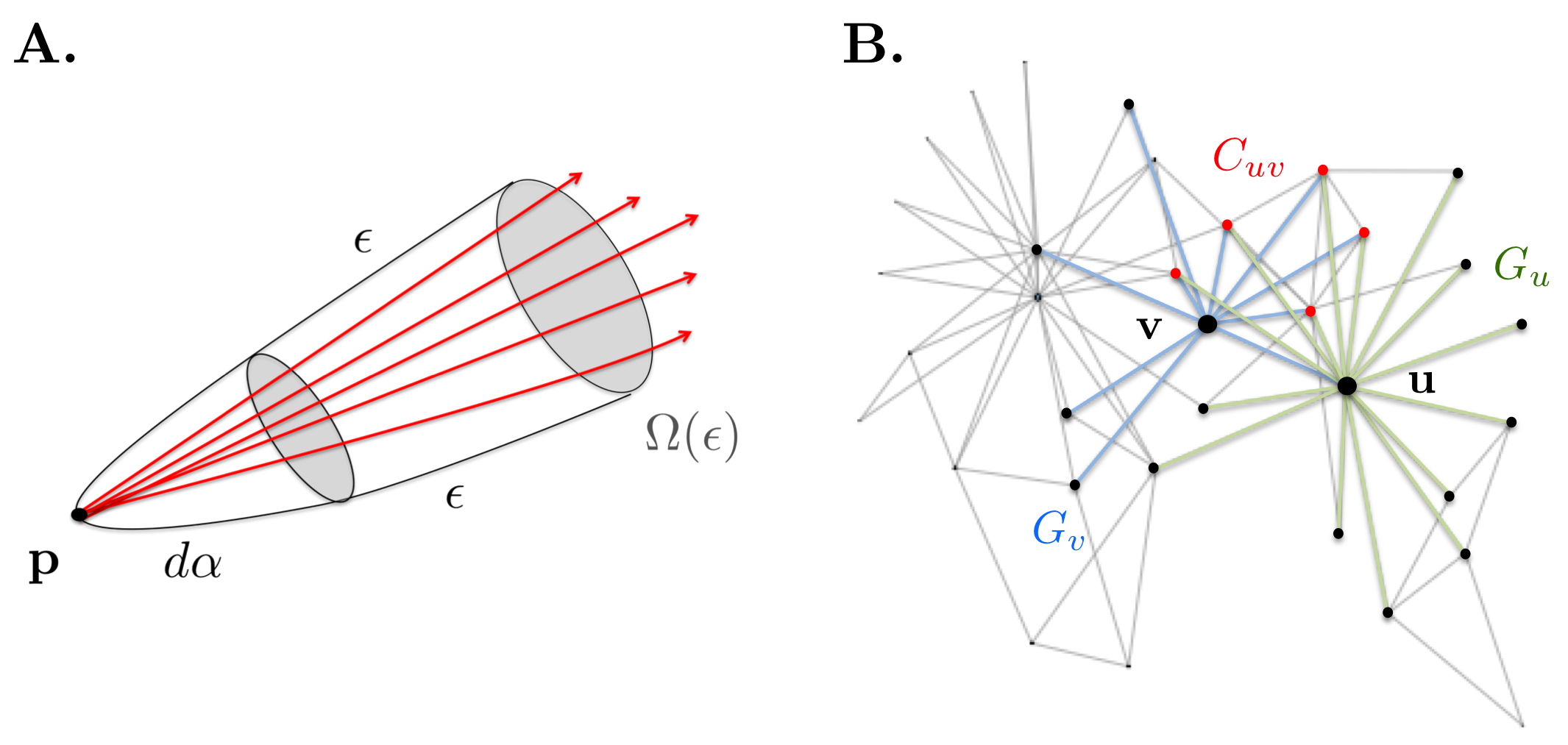} 
	\caption[Dispersion in networks]{Dispersion in networks. \textbf{A:} Dispersion of geodesics in the classic Riemannian case. \textbf{B:} Dispersion between two associated nodes in a social network (Zarachy's karate club, see \cite{karate}). The illustrated dispersion can be computed as ${\rm disp} (u,v)=5$ using Eq. (\ref{eq:disp}).}
	\label{fig:disp}
\end{figure}

\noindent To understand the close relation between the two measures, we go back to the analogy in classic Riemannian geometry. In a recent paper \cite{WJS2}, we discussed and introduced Forman-Ricci curvature as a measure of geodesic dispersal (see Fig. \ref{fig:disp}A). The dispersion Eq. (\ref{eq:disp}) gives a closely related network-theoretic counterpart to this idea. However, while closely related, we note that Eq. (\ref{eq:disp}) is essentially a measure on triangles since it represents the number of triangles that feature a given edge $e$ as a side. Therefore, in our classic analogy, this notions measures the volume (area) growth in the direction of the edge $e$, rather than the dispersion of geodesics. In consequence, dispersion is even more closely related to Ollivier's discretization of Ricci curvature \cite{Ol1}, \cite{Ol2} than to Forman's. It is furthermore related to Stone's Ricci curvature for $PL$ manifolds \cite{Stone}. 

\section{Discussion}

\noindent In this article, we introduced a higher-dimensional representation of the classic 1-dimensional graph model for networks that allows for analysis of ideal network shapes. By 'filling in' higher degree faces in sets of associated edges, we constructed 2-dimensional polyhedral complexes that allow for mapping of the discrete Ricci-formalism we introduced in \cite{WJS} to higher dimensions. In the following analysis, we mainly focused on filling in $n$-dimensional simplices.

Building on recent work by E. Bloch \cite{Bloch}, we formulated a network-theoretic Gau{\ss}-Bonnet Theorem that allows for the definition of an Euler characteristic on networks. Furthermore, we introduced a first \textit{prototype networks}. Our results illustrate the \textit{backbone-effect} of the Ricci-flow, i.e. its ability to provide an abstract of the topological and structural information of a network. It gives rise to a natural classification scheme for dynamically evolving networks based on topological properties. In contrast to previous approaches and model networks with an emphasis on node-based characteristics, the Ricci curvature and in consequence the Ricci-flow are edge-based. Defined by the information flow within the network, quantitatively described by the edge weights, both characteristics can be defined independently of node-based information overcoming the widely discussed node-degree bias present in most network-analytic tools. Furthermore, its focus on the information flow within a network makes the Ricci-formalism an ideally suited tool for the analysis of dynamics in networks.

While our study demonstrated promising results for the special case of 2-dimensional simplicial complexes, the narrow range of cases we covered can only be considered preliminary results. The emphasis of this article was to introduce theoretical tools and demonstrate their capabilities on a small set of use cases. Possible future directions include a broader study of 2-dimensional polyhedral and CW cell complexes in an attempt to generalize the prototype networks introduced in this article. In particular, one would like to give a rigorous proof (perhaps by comparing the Forman-Ricci flow introduced here to that developed by Chow and Luo) confirming that, indeed, the sign of $\overline{{\rm Ric_F}}$, determines the shape of the limiting prototype network.

A natural question in this context is, how -- if at all -- the presented formalism maps to 3-complexes. As discussed earlier, there are -- so far -- no general results or algorithmic tools to deal with singularities that are known to occur in 3-complexes. It remains an open question, if there exists a network-theoretic counterpart to G. Perelman's surgery method for the Riemannian case that might solve the singularity issue for the discrete case. \\

The present article introduced novel prototype networks in a first attempt to classify networks by counterparts of common tools from Differential Geometry: The Gau{\ss}-Bonnet theorem and Euler characteristics. We showed that one can seemingly predict the limiting case of a network's evolution based solely on its Ricci curvature. Linking back to our initial question, we \textit{can see the shape of a network} - by only evaluating its Ricci curvature.

\section{Acknowledgements}
ES thanks Moses Boudourides for bringing the notion of dispersion to his attention, and Eli Appleboim for common speculations along the lines of the title of this paper. Furthermore, he would like to thank the Max Planck Institute for Mathematics in the Sciences, Leipzig, for its support and warm hospitality. MW was supported by a scholarship of the Konrad Adenauer Foundation.

\section*{Supplemental Material}
\noindent \textbf{Supplement:} Computational implementations

\bibliographystyle{atlasnote}
\bibliography{quellen}
\end{document}